\documentclass[a4paper]{amsart}
\usepackage[margin=110pt]{geometry}
\usepackage{amssymb,amsmath,amsthm,epsf,epsfig,dsfont,bbm}
\usepackage[foot]{amsaddr}
\usepackage{bm}
\usepackage{latexsym}
\usepackage{upref, eucal}
\usepackage[normalem]{ulem} 
\usepackage[all]{xy}
\usepackage[usenames,dvipsnames,svgnames]{xcolor}
\usepackage{enumitem} 
\usepackage{comment} 
\usepackage{hyperref} 

\def\be{\kern -.1em}
\def\lbe{\kern -.025em}

\newcommand {\enm} {\ensuremath}

\newcommand{\inj}{\hookrightarrow}
\newcommand{\dd}{\delta}

\renewcommand{\AA}{\mathbb{A}}
\newcommand{\GG}{\mathbb{G}}

\newcommand{\NN}{\mathbb{N}}
\newcommand{\QQ}{\mathbb{Q}}
\newcommand{\WW}{\mathbb{W}}
\newcommand{\ZZ}{\mathbb{Z}}

\newcommand{\Ou}{\enm{\mathcal{O}}}

\newcommand{\form}{\mathrm{for}}

\newcommand{\bx}{\mathbf{x}}
\newcommand{\bxx}{\mathbf{x}}
\newcommand{\bp}{\mathbf{p}}

\newcommand{\sfp}{\mathsf{p}}

\newcommand{\bb}[1]{\mathbb{#1}}
\newcommand{\mcal}[1]{\mathcal{#1}}
\newcommand{\hG}{\widehat{\mathbb{G}}_{\mathrm{a}}}
\newcommand{\hGm}{\widehat{\mathbb{G}}_{\mathrm{m}}}

\newcommand{\C}[2]{{\binom{#1}{#2}}}

\DeclareMathOperator{\Spec}{\mathrm{Spec}}
\DeclareMathOperator{\Spf}{\mathrm{Spf}}
\newcommand{\Hom}{\mathrm{Hom}}

\newcommand{\Alg}[1]{{\mathbf{Alg}_{#1}}}
\newcommand{\aAlg}[1]{{\mathbf{Alg}^+_{#1}}}
\newcommand{\faAlg}[1]{{\mathbf{fAlg}^+_{#1}}}
\newcommand{\Nilp}{\mathbf{Nilp}}
\newcommand{\map}{\rightarrow}
\newcommand{\stk}{\stackrel}

\newcommand{\by}{{\bf y}}
\newcommand{\bz}{{\bf z}}

\newcommand{\ba}{{\bf a}}
\newcommand{\mfrak}[1]{\mathfrak{#1}}

\newcommand{\frakf}{\mathfrak{f}}
\newcommand{\tW}{W^+} 

\newcommand{\pprod}{{\prod}^+}
\newcommand{\pprodn}{{\prod}^+_n} 
\newcommand{\prodphi}{{\prod}_\phi} 
 
\newcommand{\tHom}{\mathrm{Hom}_{\aAlg{R}} }
\newcommand{\ftHom}{\mathrm{Hom}_{\faAlg{R}} }

\newcommand{\wh}{\widehat}

 \numberwithin{equation}{section}
 
 \newtheorem{theorem}[equation]{\bf{Theorem}} 
 \newtheorem{lemma}[equation]{\bf Lemma}
 
 \newtheorem{corollary}[equation]{\bf Corollary}
 \newtheorem*{theorem*}{Theorem} 
 \newtheorem*{maintheorem*}{Main Theorem}
 
 \theoremstyle{definition}
 
 \newtheorem {example}[equation]{Example}
 \newtheorem {examples}[equation]{Examples}
 \newtheorem {remark}[equation]{Remark}

\title[Canonical Witt formal scheme extensions]{Canonical Witt formal scheme extensions and $p$-torsion groups}
\date{\today}
 \author{Alessandra Bertapelle$^1$ 
 }
\address{$^1$Universit\`a degli Studi di Padova, Dipartimento di Matematica ``Tullio Levi-Civita'', via Trieste 63, I-35121 Padova, Italy}
\email{alessandra.bertapelle@unipd.it}
 \author{Nicola Mazzari$^1$ 
 }
\email{nicola.mazzari@unipd.it}
 \author{Arnab Saha$^2$
 }
\address{$^2$Indian Institute of Technology Gandhinagar, Gujarat 382355}
\email{arnab.saha@iitgn.ac.in}



\begin{document}
\baselineskip 15pt

\begin{abstract}
We study the $n$-th arithmetic jet space of the $p$-torsion subgroup attached to a smooth commutative formal group scheme. We show that the $n$-th jet space above fits in the middle of a canonical short exact sequence between a power of the formal scheme of Witt vectors of length $n$ and the $p$-torsion subgroup we started with. This result generalizes a result of Buium on roots of unity. 
\end{abstract}

\maketitle

\section{Introduction} 
\label{sec:introduction}

Buium in \cite{bui95} introduced the theory of arithmetic jet spaces on (formal) abelian schemes over $p$-adic rings and showed that the jet spaces of an abelian scheme $A$ are
naturally affine fibrations over $A$. 
Since then the theory of arithmetic jet spaces has been developed in several articles such as \cite{Barc}, \cite{bui00}, \cite{Buium-Miller}, \cite{borsah19a}, \cite{borsah19}, \cite{Hurl}, and has found remarkable applications in diophantine geometry as in \cite{bui95} and \cite{BP}.

In this paper, we study the structure of the jet space functors
 associated to the $p$-torsion subgroup 
 $G[p^\infty]$ of a smooth commutative formal group scheme $G$ over a fixed $p$-adic basis.
Here we show that for any $n$, the $n$-th jet space $J^n(G[p^\infty])$ is canonically an extension of $G[p^\infty]$ by a power of the unipotent formal group scheme $\wh{\WW}_{n-1}$, where $\wh{\WW}_{n-1}$ is $\wh{\bb{A}}^n$, the $n$-dimensional formal affine space endowed with the group scheme structure of the additive Witt vectors of length $n$. This generalizes results obtained by Buium in \cite{bui13} for $G$ the multiplicative group scheme.

Before stating our main result in detail, let us introduce some notation. Let $K$ be a finite extension of $\QQ_p$ with ramification index $e$, uniformizer $\pi$ and ring of integers $\Ou$. We denote by $k$ the residue field of $\Ou$ and let $q$ be its order. Then the identity map of $\Ou$ is a lift of $q$-Frobenius. Fix a $\pi$-adically complete $\pi$-torsion free
$\Ou$-algebra $R$ with a lifting of Frobenius $\phi$, i.e., an endomorphism of $R$ such that $\phi(r)-r^q\in \pi R$ for all $r\in R$. As an example, consider the ring of restricted power series $\Ou\langle x \rangle$ with $\phi$ the $\Ou$-algebra endomorphism given by $\phi(x) = x^q$.
Let $W_n$ be the functor of ramified Witt vectors of length $n+1$ (following Borger's convention, details in \S~\ref{s.wittvectors}).

Let $\Nilp_R$ be the category of $R$-algebras on which $\pi$ (or equivalently, $p$) is nilpotent. Its opposite category is a site with respect to the Zariski topology. Any adic $R$-algebra $A$ with ideal of definition $I$ containing $\pi$ gives rise to a sheaf (of sets) $\Spf(A)$ such that
\[
\Spf(A)(B)=\varinjlim_n \Hom_R(A/I^n,B)\ , 
\]
for any $B$ in $\Nilp_R$.
By a formal scheme over $R$ we mean a sheaf on $\Nilp_R^{\rm op}$ admitting an open cover by open subfunctors of the type $\Spf(A)$ for $A$ as above.

Given a sheaf $X$ on $\Nilp_R^{\rm op}$ we define its $n$-th $\pi$-jet by 
\begin{equation}\label{e.defJn}
J^nX(C)=X(W_n(C))
\end{equation} 
for any $C\in \Nilp_R$. If $X$ is a formal scheme over $R$ the same is $J^n X$ and it holds
\begin{equation}\label{e.adj0}
\Hom_R(\Spf(C),J^nX)=\Hom_R(\Spf(W_n(C)),X)
\end{equation}
\cite{bui95}, \cite{bor11a}, \cite{borsah19}, \cite{bps}.

For a smooth commutative formal group scheme $G$ over $R$ there is a short exact sequence of formal group schemes
\begin{equation}\label{eq:cwse}
0 \map N^nG \map J^nG \map G \map 0,
\end{equation}
which we call the \emph{canonical Witt formal scheme extension} of $G$ because of Theorem~\ref{t.Nn} below. Let $G[p^\nu]$ be the $p^\nu$-torsion 
formal subgroup scheme of $G$ and let $G[p^\infty]$ denote the sheaf on $\Nilp_R^{\rm op}$ such that
$G[p^\infty](C)=\varinjlim_\nu G[p^\nu](C)$, for any $C\in \Nilp_R$.
For each $\nu$ the closed immersions 
$G[p^\nu] \inj G[p^{\nu+1}]$ 
induce closed immersions of $\pi$-jets $J^n(G[p^\nu]) \inj J^n(G[p^{\nu+1}])$ and 
\[ J^n(G[p^\infty])= \varinjlim_\nu J^n(G[p^\nu]) \]
 as sheaves on $\Nilp_R^{\rm op}$; see Lemma \ref{l.JnGlimit}.

Consider the natural projection map $u\colon 
J^n(G[p^\infty]) \map G[p^\infty]$ 
and let $N^n(G[p^\infty])$ denote the kernel of $u$. 
Buium in \cite[Corollary 1.2]{bui13} shows that if $e=1$, $p>2$ and $G$ is the formal multiplicative group scheme over $R$, then the sheaf $N^n(G[p^\infty])$ is representable by a formal $R$-scheme and is isomorphic to $\widehat{\bb{A}}^n$, the $n$-dimensional affine space over $\Spf(R)$.

In this paper, we will enrich Buium's result and extend it to any smooth commutative formal $R$-group scheme $G$ of relative dimension $d\geq 1$.
In fact, in Theorem \ref{t.Nn} we show that if $p\geq e+2$ the kernel $N^n(G[p^\infty])$ is isomorphic to $ (\wh{\WW}_{n-1})^d$ 
where $\wh{\WW}_{n-1}$ is $\widehat{\bb{A}}^n$ endowed with the group structure of Witt vectors of length $n$.
Further, we deduce the following result (see Theorem~\ref{th.WJG}).

\begin{maintheorem*}
 Assume $p\geq e+2$. 
	Given a smooth commutative formal group scheme $G$ of relative dimension $d$ over 
	$R$, for any positive integer $n$ the natural morphism $J^nG\to G$ gives an exact sequence
	$$
	0 \map (\wh{\WW}_{n-1})^d \map J^n G[p^\infty] \map G[p^\infty] \map 0
	$$
	of sheaves on $\Nilp_R^{\rm op}$.
\end{maintheorem*}
We remark that by Lemma \ref{l.JnGlimit} it is $J^n(G[p^\infty])=(J^n G)[p^\infty]$ as sheaves on $\Nilp_R^{\rm op}$; hence there is no possible ambiguity in the above statement.

\subsection{Plan of the paper}
In Section~\ref{sec:pjets facts} we recall the definition and properties of $\pi$-jets in the setting of formal schemes, with particular attention to the adjunction between jet algebras and Witt vectors \eqref{e.adj}.

In Section~\ref{s.jetkernel} we focus on the notion of shifted Witt vectors $W_n^+$, introduced by \cite{borsah19}, and show that $W_n^+$ induces an adjoint functor to $N^n$ (Theorem \ref{t.adj0}). This is an important result, analogous to the adjunction formula that involves $W_n$ and $J^n$ \eqref{e.adj0}. Then we show that given a smooth formal $R$-group scheme $G$ of dimension $d$, we have 
\begin{itemize}[nolistsep]
\item[i)] \emph{ For all $n>0$, $N^nG \simeq J^{n-1}(N^1G)$;} see Theorem~\ref{t.Njet}.
	\item[ii)] \emph{Assume $p\geq e+2$. Then there is a natural isomorphism of formal group schemes }
$N^n G\simeq (\wh{\WW}_{n-1})^d\ $; see Theorem~\ref{t.Nn}. 
\end{itemize} 
The proof of the first fact reduces to a local computation in coordinates which is detailed in the Appendix section. Another important ingredient is the notion of \emph{lateral Frobenius} introduced in \cite{borsah19}. 
Both results i) and ii) are generalized to the case of $m$-shifted Witt vectors in \cite{Sah} by the third author. 

In Section~\ref{sec:tors} we apply the previous results to the study of the sheaves $J^n(G[p^\infty])$ and $N^n(G[p^\infty])$ and deduce a statement similar to Theorem~\ref{t.Nn} where $G$ is replaced by $G[p^\infty]$, see Theorem~\ref{th.WJG}.

In this paper all rings are assumed to be commutative with unit and $\Alg{R}$ denotes the category of $R$-algebras, i.e., of ring homomorphisms $R\to B$.

\section{Arithmetic jets}\label{sec:pjets facts}

\subsection{Conventions} Let $R$ be the base ring fixed in the introduction. Given a formal scheme $X$ and a fixed point $x\colon\Spf(R)\to X$, one can consider the fibre of $x$ under the natural map 
$J^nX\to X$,
which is the closed formal subscheme $N^nX=(J^nX)_x= J^nX\times_X\Spf(R)$. Note that
if $G$ is a formal group scheme then $J^nG$ is naturally a formal group scheme too and we set $N^nG=(J^nG)_\varepsilon=\ker(J^nG\to G)$ to be the fibre along the unit section $\varepsilon$.

If $X$ is a functor on $\Alg{R}$, we will usually denote by $\widehat{X}$ the restriction of $X$ to $\Nilp_R$. Let $R\langle x_1,...,x_n\rangle$ be the $\pi$-adic completion of the $R$-polynomial algebra in $n$ variables. 

Let $\hG:=\Spf(R\langle x\rangle)$ be the additive formal group scheme over $R$. Note that this formal group scheme should not be confused with the $(x)$-adic formal group $\GG_a^{\rm for}=\Spf(R[[ x]])$, the formal completion of $\GG_a$ along the zero section.

If $\mathcal F\in R[[\bx,\by]]$ is a commutative formal group law of dimension $g$, let ${\mathcal F} \{n\}$ be the formal group law given by $\pi^{-n}\mathcal F(\pi^n \bx,\pi^n \by)$, for any $n\geq 1$. Note that ${\mathcal F} \{n\}$ endows $\Spf(R\langle \bx\rangle)$ with a structure of formal group scheme over $R$.

If $B$ is an $R$-algebra, $\rho=\rho_B\colon R\to B$ always denote the corresponding ring homomorphism. If the context is clear, we will write $r$ in place of $\rho(r)\in B$.

\subsection{Witt vectors over $R$}\label{s.wittvectors}
In the following pages $W_n$ denotes the functor of $\pi$-typical Witt vectors of length $n+1$ on $R$-algebras. Hence, for any $R$-algebra $B$, the ring $W_n(B)$ is always considered with its natural $R$-algebra structure, which depends on $\phi$. We explain this briefly.

As functor on $\Ou$-algebras $W_n$ coincides with the so-called functor of ramified Witt vectors of length $n+1$ (see \cite{haz}, \cite{bc}). Let $w\colon W_n(R) \map \prod_{i=0}^n R$ be the ghost map. Then for any Witt vector $\ba=(a_0,\dots, a_n)$, $w(\ba)=(w_0(\ba), \dots, w_n(\ba))$ where $w_i$ are the ghost polynomials
\begin{equation}\label{e.ghost}
w_i = x_0^{q^i}+ \pi x_1^{q^{i-1}}+ \cdots + \pi^i x_i.
\end{equation}

Since $R$ has a lifting of Frobenius $\phi$, by the universal property of Witt vectors there exists a ring homomorphism of $\Ou$-algebras $\exp_{\delta}$ making the following diagram commute (see \cite[(2.9)]{bps})
\begin{equation}\label{d.rwr}
\xymatrix{
	R\ar[rr]^{\exp_\dd} \ar[drr]_{(\phi^0,\phi,\dots,\phi^n)\ } && W_n(R)\ar[d]^w \\
& & \prod_{i=0}^{n} R 
}
\end{equation}

Let $B$ be an $R$-algebra.
Then $W_n(B)$ is naturally endowed with the $R$-algebra structure 
\[R\stk{\exp_{\dd}}{\longrightarrow} W_n(R)\stk{W_n(\rho_B)}{\longrightarrow} W_n(B).
\]

In \cite[\S 3.2]{borsah19} the authors give an equivalent construction of the functor $W_n$. 
The ghost map $w$ in \eqref{d.rwr} is $\Ou$-linear, but not $R$-linear in general, if the ring $\prod_{i=0}^{n} R$ is endowed with the direct product $R$-module structure. It is then preferable to change the $R$-module structure on the product ring so that $w$ becomes $R$-linear. 
Let ${}^{\phi^n}\!B$ denote the ring $B$ with the $R$-algebra structure induced by $\rho_B\circ \phi^n\colon R\to B$, and let 
\begin{equation}\label{e.prodphi}
\prodphi^n (B):=\prod_{i=0}^n ({}^{\phi^i}\!B)
\end{equation}
be the direct product algebra. Its underlying ring is $\prod_{i} B$ and there is a commutative diagram of $R$-algebras
\begin{equation}\label{d.rwr2}
\xymatrix{
	R\ar[rr]^\varphi \ar[drr]_{({\rm id},\phi,\dots,\phi^n)\ } && W_n(R)\ar[d]^w \ar[rr]^{W_n(\rho_B)}&& W_n(B)\ar[d]^w \\
	& & \prodphi^n(R) \ar[rr]^{\prod_i(\rho_B)}&&\prodphi^n(B)
}
\end{equation}
Then Frobenius and Verschiebung maps can be described in terms of ghost components as in the case of ramified Witt vectors, with caution when considering the $R$-algebra structure. As for example, the Frobenius ring homomorphism $F\colon W_n(B)\to W_{n-1}(B)$ described in terms of ghost components as the left shift is $\phi$-semilinear. We prefer then to write it as the homomorphism of $R$-algebras
\begin{equation}\label{e.F}
F\colon W_n(B)\to W_{n-1}({}^\phi\!B)
 \end{equation}
corresponding to the homomorphism of $R$-algebras 
 \begin{equation}
 \label{e.Fw}
 F_w\colon \prodphi^n\! (B)\to \prodphi^{n{\text{--}}1}\! ({}^\phi\!B), \quad (b_0,\dots, b_n)\mapsto (b_1,\dots,b_n).
 \end{equation}
Similarly, the Verschiebung map $V\colon W_n(B)\to W_{n+1}(B)$ is described on ghost components as the right shift multiplied by $\pi$. Clearly it is $\Ou$-linear but not $R$-linear in general. We prefer then to write it as the homomorphism of $R$-modules 
\begin{equation}\label{e.V}
V\colon W_n({}^\phi\!B)\to W_{n+1}(B)
\end{equation}
corresponding to the homomorphism of $R$-modules
\begin{equation}
\label{e.Vw}
V_w\colon \prodphi^n\! ({}^\phi\!B)\to \prodphi^{n{\text{+}}1}\! (B), \quad (b_0,\dots, b_n)\mapsto (0,\pi b_0,\dots,\pi b_n).
\end{equation}
Then $FV$ is multiplication by $\pi$ on $W_n({}^\phi\!B)$.

Since $\phi$ might not be invertible, one can not write $B$ in place of $ {}^\phi\!B$ in \eqref{e.F} and \eqref{e.V}. 
However, since $\phi$ is the identity on $\Ou$, the $\Ou$-module structure on $B$ and ${}^\phi\!B$ are the same.

\begin{remark}\label{r.piadicW}
If $B$ is a $\pi$-adic $R$-algebra (by this we mean $\pi$-adically complete and separated) then the same is $W_n(B)$ for any $n$. The proof works as in \cite[Proposition 3]{zin}.
\end{remark}

\subsection{Shifted Witt vectors} \label{s.nuclearwitt} 
We recall the construction of $0$-shifted Witt vectors as introduced in 
\cite{borsah19} and \cite{Sah}. Here we simply refer to them as shifted Witt vectors. The general theory of $m$-shifted Witt vectors is developed in \cite{Sah}.

Let $B$ be an $R$-algebra and set-theoretically 
define 
\begin{align}
\label{e.tWn}
\tW_n(B):=R \times_B W_n(B) \simeq R \times W_{n-1}(B).
\end{align} 
Also define the product ring 
	$$\pprodn(B) := R \times \prod_{i=1}^n ({}^{\phi^i}B)= R\times {\prod}_{\phi}^{n-1}({}^{\phi}\!B) ,
	$$
	where $\prodphi^n(B)$ was introduced in \eqref{e.prodphi}. Note that there is an isomorphism of $R$-algebras 
	\begin{equation*} 
	\pprodn B:= R\times {\prod}_{\phi}^{n-1}({}^{\phi}\!B) \simeq 
	R\times_B {\prod}_{\phi}^{n}(B)
	\end{equation*}
	mapping $(r,b_1,\dots,b_n)$ to the element $\left(r,(\rho(r),b_1,\dots,b_n)\right)$.

Define the \emph{ apriori} set-theoretic map $w^+\colon \tW_n(B) \map \pprodn(B)$ given
by $w^+(r,b_1,\dots, b_n) = (w_0,\dots, w_n)$ where 
$$w_i = r^{q^i}+\pi b_1^{q^{i-1}} + \dots + \pi^i b_i,
$$
for all $i=0,\dots, n$. Then note that 
$\tW_n(B)$ naturally is endowed with the Witt ring structure of addition and 
multiplication making $w^+$ a ring homomorphism. Hence we have 
the following commutative diagram
\begin{equation}\label{d.plus}
\xymatrix{
	\tW_n(B)\ar[rr]^{w^+} \ar@{}[drr]|\square\ar[d] &&\pprodn(B)\ar[rr]^{{\rm pr}} \ar@{}[drr]|\square \ar[d]&& R\ar[d]^{\rho} \\
	W_n(B)\ar[rr]^w && \prodphi^n(B)\ar[rr]^{{\rm pr}_0} &&B
}
\end{equation}
where ${\rm pr}_0$ is the projection onto the $0$-th component.
The $R$-algebra $\tW_n(B)$ was denoted by $\tilde W_n(B)$ in \cite[\S 4]{borsah19} and by $W_{0n}(B)$ in \cite{Sah}.

Since the lower horizontal arrows in \eqref{d.plus} are homomorphisms of $R$-algebras, the same are the upper horizontal arrows. Hence the left hand square in \eqref{d.plus} is a diagram of $R$-algebras and, up to the above identifications, it can be illustrated as 
\begin{equation}\label{d.plus2}
\xymatrix{
	R\times W_{n-1}({}^\phi\!B)\ar[r]^{w^+} \ar[d] &R\times \prodphi^{n\text{--}1}({}^\phi\!B) \ar[d] \\
	W_n(B)\ar[r]^w & \prodphi^n(B)}\quad 
\xymatrix@C-10pt{
	\big(r,(b_1,\dots,b_n)\big)\ar[r]\ar[d]& \big(r,(w_1(r,b_.),\dots,w_n(r,b_.))\big) \ar[d] \\
	(r,b_1,\dots,b_n) \ar[r] & (r, w_1(r,b_.),\dots,w_n(r,b_.))} 
\end{equation}
where we have written $r$ in place of $\rho(r)$ in $B$ and 
$w_i$ are the ghost polynomials in \eqref{e.ghost}.

\subsection{Prolongation sequences} For any formal schemes $Y$ and $Z$ over $\Spf(R)$ we say that $(u,\dd)\colon Z \map Y$ is a prolongation if $u\colon Z \map Y$ is a morphism of formal schemes over $\Spf(R)$ and $\dd\colon \Ou_Y \map u_* \Ou_Z$ is a $\pi$-derivation on the sheaves (cf. Appendix \ref{sec:app}).
Then a sequence of formal schemes $T^*=\{T^n\}_{n=0}^\infty$ is a \emph{ prolongation sequence} if for each $n$, $(u_n,\dd_n)\colon T^{n+1} \map T^n$ is a prolongation of formal group schemes over $\Spf(R)$ satisfying $u_{n-1}^* \circ \dd_n = \dd_{n-1} \circ u_n^*$ and making the following diagram commute
$$
\xymatrix{
	R \ar[r] & u_* \Ou_Z \\
	R \ar[u]^\dd \ar[r] & \Ou_Y\ . \ar[u]_\dd
}
$$
A morphism of prolongation sequences $T^*\to P^*$ is a system of morphisms of formal schemes $f_n\colon T^n\to P^n$ that satisfy the expected commutations: $f_n \circ u_{n}= u_{n}\circ f_{n+1}$ and $f_n \circ \dd_{n}= \dd_{n}\circ f_{n+1}$.
	For each $n$, let $S^n = \Spf(R)$. Then the fixed $\pi$-derivation $\dd$ on $R$ makes $S^*$ into a prolongation sequence. Let $\mcal{C}_{S^*}$ denote the category of prolongation sequences defined over $S^*$.

\subsection{Jet spaces}\label{s.jet} 
Given a formal scheme $X$ over $S^0=\Spf(R)$, Buium constructs the \emph{ canonical prolongation sequence} $J^*X= \{J^nX \}_{n=0}^\infty$ where $J^0X =X$ and by \cite[Proposition~1.1]{bui00}, $J^*X$ satisfies the following universal property: for any $T^*$ in $ \mcal{C}_{S^*}$ it is
\begin{align}
 \Hom_{S^0} (T^0, X) = \Hom_{\mcal{C}_{S^*}}(T^*,J^*X)\ .
\end{align}
Moreover, by \cite{bor11b} and \cite{bps} we have the following functorial description: 
\begin{equation*}
J^nX(C) = X(W_n(C))
\end{equation*} 
for any $C\in \Nilp_R$. In particular in the affine case with $X=\Spf(A)$ and $J^nX=\Spf(J_nA)$, we have a natural adjunction
\begin{equation}\label{e.adj}
\Theta\colon \Hom_R(J_nA, C) \stackrel{\sim}{\longrightarrow} \Hom_R(A,W_n(C))
\end{equation} 
such that $w_0\circ \Theta(g)=g\circ \iota$ with $\iota\colon A\to J_nA$ the natural morphism. 

Here we make the above adjunction explicit when 
$X = \Spf(R\langle x\rangle)$ is the formal affine line over $\Spf(R)$.
 Let $A = R\langle x \rangle$. Then $J^nX = \Spf(J_nA)$ and
\begin{equation}\label{e.p}
	J_nA = R\langle x, x', \dots x^{(n)}\rangle = R\langle \sfp_0,\dots \sfp_n \rangle
\end{equation}
where $x,x',\dots , x^{(n)}$ are the Buium-Joyal coordinates and $\sfp_0,\dots \sfp_n$ are the Witt coordinates and they satisfy $\sfp_0 =x$, $\sfp_1 = x'$ while the general relation between the above two coordinate systems can be found in \cite[Proposition~2.10]{bps}. 

If $g \in \Hom_R(J_nA,C)$, then 
$\Theta(g) \in \Hom_R(A, W_n(C))$ is determined by 
by 
\[ \Theta(g)(x) = (g(\sfp_0),\dots , g(\sfp_n))\ . \]

Note that when $G = \hG$ we have the following isomorphism of formal group schemes
\begin{equation}\label{e.jetGa}
J^n \hG=\widehat{\bb W}_n.
\end{equation}
where $\widehat{\bb{W}}_n$ is $\widehat{\bb{A}}^{n+1}$ endowed with the 
additive group structure of Witt vectors of length $n+1$.

\section{The Kernel as a $\pi$-jet space}\label{s.jetkernel}
For any sheaf of groups $G$ on $\Nilp_R^{\rm op}$
one defines 
\[N^nG:=\ker(J^nG\stackrel{u}{\longrightarrow} G)\] 
where $u$ is the natural morphism.
Scope of this section is to take a closer look at the kernel $N^nG$ in the case $G$ is representable by a smooth formal scheme. Since the kernel is the fibre at the unit section, we will first consider more general fibres.
 
 Let $X$ be a smooth formal scheme over $R$ with a marked point $a$ and let $u\colon U\to \AA=\Spf(R\langle\bx\rangle)$ be an \'etale chart around $a$, where $\bx$ denotes here a finite family of indeterminates. Hence $U$ is an open affine formal subscheme of $X$, $u$ is \'etale, $a$ factors through $u$ and $u\circ a$ is the zero section $0$ of the affine space $\AA$. 
 By \cite[Proposition~3.13 \& Corollary~3.16]{bui05} (see also \cite[Proposition 3.12]{bps}), we have $	J^nU \simeq J^n\AA \times_\AA U $ for all $n$ and hence 	
 \begin{equation}\label{e.NXNA}
 N^nX= J^nX\times_{X,a}\Spf(R)=J^nU\times_{U,a}\Spf(R)=J^n\AA\times_{\AA,0}\Spf(R)=N^n\AA\ ;
 \end{equation}
 in particular, $N^nX$ is formal affine and isomorphic to $N^nU$. Up to shrinking $U$, we may assume $U= \Spf(A)$ with $A$ a $\pi$-adically complete separated $R$-algebra. Then the point $ a\colon \Spf(R)\to U$ induces an $R$-algebra morphism $\varepsilon\colon A \map R$. 
 Let $J^nU = \Spf(J_nA)$.
 It is immediate to check that
 \begin{equation}\label{eq:Nn}
 	N^nU = \Spf(J_nA\hat\otimes_{A,\varepsilon} R).
 \end{equation}

\subsection{Adjunction}\label{s.functNn}
Let $\aAlg{R}$ denote the category of augmented (commutative) $R$-algebras. Its objects are commutative $R$-algebras $A$ together with an augmentation $\varepsilon$, i.e. an $R$-algebra morphism $\varepsilon\colon A\to R$; morphisms in $\aAlg{R}$ are morphisms of $R$-algebras $h\colon A_1\to A_2$ respecting augmentations, i.e. $\varepsilon_2\circ h=\varepsilon_1$.
For any $(A,\varepsilon)$ in $\aAlg{R}$ we define the $R$-algebra \begin{equation}\label{e.Nn}
N_nA:= J_nA\otimes_{A,\varepsilon} R.
\end{equation}
Note that shifted Witt vectors yield objects in $\aAlg{R}$. Indeed, let $B$ be an $R$-algebra and let $w_0^+\colon W_n^+(B)\to R$ denote the projection onto the first component, i.e., the composition of the upper horizontal arrows in \eqref{d.plus}. Then $W_n^+(B)$ together with $w_0^+$ is an augmented $R$-algebra. Hence the above construction defines a functor
\[W_n^+\colon \Alg{R}\longrightarrow \aAlg{R}\ , \quad B\mapsto (W_n^+(B),w_0^+) \ ,\]
on the category of $R$-algebras.

We now prove a key result: $N_n$ and $W_n^+$ is a pair of adjoint functors.

\begin{theorem}\label{t.adj}
For any augmented $R$-algebra $(A,\varepsilon)$ and any $R$-algebra $B$ there is a natural bijection
\begin{equation}\label{e.phiplus}
\Theta^+\colon \Hom_R(N_nA,B)
\stackrel{\sim}{\longrightarrow}\tHom(A,\tW_n(B)) \ .
\end{equation}	
\end{theorem}

\begin{proof}
Let $\iota\colon A\to J_nA$ denote the natural morphism. Then 
\begin{align*}
\Hom_R(N_nA,B)=& \{g \in\Hom_R(J_nA,B)\ |\ g\circ \iota =\rho_B\circ \varepsilon\}
\\
=&\{f\in\Hom_R(A,W_n(B))\ |\ w_0\circ f=\rho_B\circ \varepsilon\} \\
=&\{f^+ \in\Hom_R(A,\tW_n(B))\ |\ w_0^+\circ f^+ =\varepsilon\}\\
=&\tHom(A,\tW_n(B)) 
,
\end{align*}
where the first equality follows by \eqref{e.Nn}, the second by \eqref{e.adj} taking $f=\Theta(g)$, the third by definition of $W_n^+(B)$ in \eqref{e.tWn}. 
\end{proof}

By Remark \ref{r.piadicW} an analogous adjunction holds when working with the category of augmented formal $R$-algebras $\faAlg{R} $. We make this fact explicit.
\begin{example}\label{r.formal}
Let $A=R\langle x\rangle$ and $\varepsilon(x)=0$. By \eqref{e.p}
	\begin{equation*}\label{e.adj2} N_nA=J_nA\hat\otimes_{A,\varepsilon} R= R\langle x',\dots, x^{(n)}\rangle=R\langle\sfp_1^+,\dots,\sfp_n^+\rangle,
	\end{equation*}
where $\sfp_i^+ \in R[x',\dots, x^{(n)}]$ denotes the polynomial $\sfp_i\in R[x ,x',\dots, x^{(n)}]$ evaluated at $x=0$. Then the formal counterpart of \eqref{e.phiplus}
works as follow: given $g\in \Hom_R(N_nA,B)$, then $\Theta^+(g)$ maps $x$ to $(0,g(\sfp_1^+),\dots ,g(\sfp_n^+))$.

The higher dimensional case is analogous. 
Let $A=R\langle\bx\rangle$ with $\bx $ a collection of $r$ indeterminates
$\{x_{ 1},\dots ,x_{ r}\}$ and let $\varepsilon$ be the zero section. Then $J_nA\simeq R\langle\bp_0,\bp_1,\dots , \bp_n\rangle$ where 
$\bp_i$ denotes a collection of polynomials
$\{\sfp_{i,1},\dots ,\sfp_{i,r}\}$ and $\sfp_{i,j}\in R[x_j,x_j',...,x_j^{(n)}]$ plays the role of $\sfp_i$ in \eqref{e.p}.
Then
\begin{equation}\label{e.Nnp}
N_n A \simeq R\langle\bp_0,\dots \bp_n\rangle
{\hat \otimes}_{R\langle\bx\rangle,\varepsilon} \ R \simeq R\langle\bp_1^+,\dots , \bp_n^+\rangle,
\end{equation}
where $\bp_i^+$ denotes the collection of polynomials
$\{\sfp_{i,1}^+,\dots ,\sfp_{i,r}^+\}$ with $\sfp_{i,j}^+$ obtained by evaluating $\sfp_{i,j}$ at $x_j=0$.
Finally for a homomorphism $g\colon N_nA\to B$, $\Theta^+(g)$ maps $x_i$ to $(0,g(\sfp_{1,i}^+),\dots ,g(\sfp_{n,i}^+))$.
\end{example}

We can now describe the functor $N^n X$ on $R$-algebras as done in \eqref{e.adj0} for $J^nX$. 
\begin{theorem}
	\label{t.adj0}
Let $X$ be a smooth formal scheme over $R$ with a marked point $x$ and let $B$ be in $\Nilp_R$. Then 
\[ N^n X(B)=\Hom_{R\text{\rm -pt}}(\Spf(W_n^+(B)), X),
\]
where on the right we are considering morphisms of $R$-pointed formal schemes.
\end{theorem}
\begin{proof}
The result is clearly true if $X$ is affine by Theorem \ref{t.adj}. For the general case, assume first that $X$ is an $R$-scheme and consider the following diagram
	\[\xymatrix{ && \Spec(B)\ar[ddll]_g \ar[d]^\rho\ar[r]^(0.4){w_0}& \Spec(W_n(B))\ar@{.>}[d]\ar[ddrr]^f &&\\
			&&\Spec(R) \ar[drrr]_x\ar@{.>}[r]^(0.4){w^+_0}& \Spec(W_n^+(B)) \ar[drr]^(0.44){f^+}&&\\
	J^nX \ar[rrrrr]^u	&&&&&X
	}\]
where $ w_0$ is induced by the projection on the first component on algebras, $\rho$ is the structure morphism and $u$ is the natural map. Note that $\Spec(W_n^+(B))$ is the push-out of $w,\rho$ in the category of \emph{ all} schemes \cite[07RS]{st}. Then $N^nX(B)=X(W_n^+(B))$. Indeed,
 \begin{align*}
N^nX(B)=& \{g \in\Hom_R (\Spec(B),J^nX )\ |\ u\circ g =x\circ \rho \}
\\
=&\{f\in\Hom_R(\Spec (W_n(B)), X)\ |\ f\circ w_0=x\circ \rho\} \\
=&\{f^+ \in\Hom_R(\Spec (W_n^+(B)),X)\ |\  f^+\circ w^+_0 =x\}\\
=&\Hom_{R\text{\rm -pt}}(\Spec(W_n^+(B)), X).
\end{align*} 	
If $X$ is a formal scheme, then the above holds for all schemes $X \times \Spec R/(\pi^m)$ and thus one concludes. 
\end{proof}

\subsection{A special case}
Let $G$ be a formal group scheme over $R$ and denote by $G^\form$ the formal completion of $G$ along the unit section. 
Let $\mathcal F\in R[\![x_1,\dots, x_r,y_1,\dots, y_r]\!]$ be the formal group law on $ G^\form$, ${\mathcal F}^\phi$ the one obtained by acting on the coefficients of $\mathcal F$ by $\phi$ and ${\mathcal F}^\phi\{1\} :=\pi^{-1}\mathcal F^\phi(\pi \bx,\pi \by)$, where $\pi\bx:=(\pi x_1,\dots, \pi x_r)$.
By \cite[Lemma 2.2]{bui95} it is $N^1G\simeq {\mathcal F}^\phi\{1\}$ as formal $R$-schemes. 
We give below a direct computation of this fact.

\begin{lemma}\label{l.N1Gfgrouplaw}
	Let the notation be as above. Then the formal group law on the formal completion of $N^1G$ at the origin is isomorphic to ${\mathcal F}^\phi\{1\}$.
\end{lemma}
\begin{proof}
As seen in Remark \ref{r.formal}, we may write
$G^\form= \Spf(R[\![\bx]\!])$ and $N^1G=\Spf(R\langle \bp_1^+\rangle)=\Spf(R\langle \bx'\rangle)$. Let $\delta\colon R[\![\bx,\by]\!]\to R[\![\bx,\by,\bx',\by']\!]$ be the $\pi$-derivation compatible with that of $R$ and such that $\delta(\bx)=\bx'$, $\delta(\by)=\by'$.
If $\mathcal F(\bx,\by)$ is the formal group law of $G$, the formal group law of $N^1G$ is $\delta(\mathcal F(\bx,\by))$ evaluated at $\bx=\bm 0,\by=\bm 0$.
Write $\mathcal F(\bx,\by)=\sum_{\alpha,\beta} a_{\alpha,\beta} \bx^\alpha \by^\beta$ with $\alpha,\beta$ varying in $\NN^r \smallsetminus \{\bm 0\}$, $\bx^\alpha:=x_1^{\alpha_1}x_1^{\alpha_2}\dots x_r^{\alpha_r}$, and the coefficients of the monomials of degree $1$ equal to $1$.
By induction, applying the usual rules of $\pi$-derivations, one checks that 
\begin{multline*}\delta(\mathcal F(\bx,\by))_{|\bx=\bm 0, \by=\bm 0}=\sum_{\alpha,\beta} \delta(a_{\alpha,\beta} \bx^\alpha\by^\beta )_{|\bx=\bm 0, \by=\bm 0}=\\ \sum_{\alpha,\beta} \pi^{-1}\phi(a_{\alpha,\beta}) (\pi\bx')^\alpha(\pi\by')^\beta=\pi^{-1}\mathcal F^\phi(\pi\bx,\pi\by)=\mathcal F^\phi\{1\}.
\end{multline*}
\end{proof}

\subsection{Lateral Frobenius}\label{s.latfrob} 
Let $X$ be a formal $R$-scheme. As $n$ varies, the $\pi$-jet spaces $J^nX$ form an inverse system of formal schemes and, more precisely, a prolongation sequence, whence a lifting of Frobenius $\phi_J$ exists on the limit. Clearly, the transition maps $u=u^{n+1}_n\colon J^{n+1} X\to J^nX$ induce homomorphisms $N^{n+1} X\to N^nX$, but the image of $\phi$ restricted to 
$N^{n+1}X$ is not necessarily contained in $N^nX$ and hence $\phi$
 does not induce a lifting of Frobenius on the sequence of the kernels. For this reason, the notion of \emph{lateral Frobenius} was introduced and studied in \cite{borsah19,Sah}.

On shifted Witt vectors, the lateral Frobenius
$F^+$ is defined as the homomorphism of $R$-algebras making the following diagram
\begin{equation}\label{e.Fplus}
\xymatrix{
\tW_n(B)\ar@{=}[d]\ar[r]^{F^+}& \tW_{n-1}({}^\phi\!B)\ar@{=}[d] \\
R\times W_{n-1}({}^\phi\!B)\ar[r]^{{\rm id}\times F}& R\times W_{n-2}({}^{\phi^2}\!B), 
} 
\end{equation}
commute, where vertical identifications are meant as sets, and $F$ is the Frobenius on Witt vectors recalled in \eqref{e.F}. The homomorphism $F^+$ then corresponds to the homomorphism of $R$-algebras 
\begin{equation}
\label{e.Fwplus}
F_w^+\colon R\times \prodphi^{n{\text{--}}1}\! ({}^{\phi}\!B)\to R\times \prodphi^{n{\text{--}}2}\! ({}^{\phi^2}\!B), \quad (r,b_1,\dots, b_n)\mapsto (r,b_2,\dots,b_n)
\end{equation}
via ghost map, i.e., it makes the following diagram 
\begin{equation}\label{d.plus3}
\xymatrix{
	\tW_n(B)\ar[rr]^-{w^+} \ar[d]_{F^+} && \ar[d]_{F^+_w} \pprodn(B)= R\times \prodphi^{n{\text{--}}1}\! ({}^\phi\!B) \\
	\tW_{n-1}({}^\phi\!B)\ar[rr]^-{w^+} && \pprod_{n{\text{--}}1}\! ({}^\phi\!B)=R\times \prodphi^{n{\text{--}}2}\! ({}^{\phi^2}\!B)
}
\end{equation}
commute.
Then, the homomorphism $F^+$ induces via \eqref{e.adj} a natural morphism
 \[\frakf\colon N^nX\to N^{n-1}X \]
 called again \emph{lateral Frobenius}. It is showed in 
\cite[Theorem 4.3]{borsah19} that $\frakf$ is a lift of Frobenius 
and satisfies 
\[\phi\circ \phi\circ u=\phi\circ u\circ \frakf,\]
where $u$ denotes the immersion $N^mX\to J^mX$ and $\phi$ denotes the Frobenius morphism $J^mX\to J^{m-1}X$ for any $m$. 
\begin{remark}\label{r.wFplus}
For later use, note that the element $(0, b_.)=(0,b_1,\dots,b_n)\in 	\tW_n(B)$ traces in \eqref{d.plus3} the following images
\[
\xymatrix{
	(0,b_{1},\dots ,b_{n}) \ar@{|-{>}}[d] 
	\ar@{|-{>}}[r] &
	(0,\pi b_{1},\pi b_{1}^{q}+ \pi^{2}b_2,\dots,\pi w_{n-1}(b_.) ) \ar@{|-{>}}[d] \\
	(0, c_1,\dots, c_{n-1} )\ar@{|-{>}}[r] & (0,
	\pi b_{1}^{q}+ \pi^{2}b_2,\dots, \pi w_{n-1}(b_.) )\ .
}
\]	
Hence $\pi c_1=\pi b_{1}^{q}+ \pi^{2}b_2$ implies $c_1=b_{1}^{q}+ \pi b_2$ (for $B$ without $\pi$-torsion and hence for any $R$-algebra $B$ by standard arguments). By recursion one sees 
\begin{equation*}
(F^+)^i(0,b_{1},\dots ,b_{n})=(0,b_1^{q^i}+\pi b_2^{q^{i-1}}+\dots+\pi^i b_i, \dots )\in \tW_{n-i}(B)
\end{equation*}
for any $i<n$.
\end{remark}

If $G$ is smooth over $R$, the same are $J^nG$ and $N^nG$ for all $n$. As seen in the previous section, $N^nG=\Spf(N_n A)$ is an affine space over $R$.	In particular the $R$-algebras $N_nA$ are $\pi$-torsion free and therefore the lateral Frobenius homomorphisms $\mfrak{f}^* \colon N_nA \map N_{n+1}A$ induce a unique $\pi$-derivation $\Delta$ on the prolongation sequence $N_*A:=\{N_nA\}_{n=1}^\infty$.

In order to describe $N^n$ as a jet functor we need a preparation lemma.

\begin{lemma}	\label{l.fi}
Let $\AA= \Spf(R\langle\bx \rangle)$ with $\bx $ a collection of $r$ indeterminates, and choose the origin as marked point. Let $\frakf^i\colon N^n\AA\to N^{n-i}\AA$ denote the $i$-th fold composition of lateral Frobenius for any $i\leq n$. 
Then $\frakf^i$ induces an homomorphism of $\pi$-adic $R$-algebras
	\[ (\frakf^i)^*\colon R\langle \bp^+_1,\dots , \bp^+_{n-i}\rangle \longrightarrow R\langle \bp^+_1,\dots , \bp^+_n\rangle \]
such that 
\[	(\mfrak{f}^i)^*(\bp^+_1) = (\bp^+_1)^{q^{i-1}}+ \pi (\bp^+_2)^{q^{i-2}} + \cdots + \pi^{i-1}\bp^+_i.
\] 
\end{lemma}
\begin{proof}
Recall from Example~\ref{r.formal} that $N^n\AA=\Spf(N_nR\langle\bx \rangle) \simeq \Spf(R\langle \bp^+_1,\dots , \bp^+_n\rangle)$. 
Then by definition of $\mfrak{f}$ and \eqref{e.phiplus} with $B= N_n R\langle\bx \rangle$ we have a commutative diagram of rings
\begin{equation*} 
\xymatrix{
\Hom_R(R\langle \bp^+_1,\dots , \bp^+_n\rangle,R\langle \bp^+_1,\dots , \bp^+_n\rangle) \ar[d]^{-\circ (\mfrak{f}^*)^i} \ar[r]^-{\Theta^+}_\sim&	\ftHom(R\langle\bx \rangle ,\tW_n(R\langle \bp^+_1,\dots , \bp^+_n\rangle))
\ar[d]^{(F^+)^i\circ -}\\
\Hom_R(R\langle \bp^+_1,\dots , \bp^+_{n-i}\rangle,R\langle \bp^+_1,\dots , \bp^+_n\rangle) \ar[r]^-{\Theta^+}_\sim &\ftHom(R\langle\bx \rangle ,\tW_{n-i}(R\langle \bp^+_1,\dots , \bp^+_{n}\rangle))
}
\end{equation*}
By Remark \ref{r.wFplus} the identity map on $R\langle \bp^+_1,\dots , \bp^+_{n}\rangle$ traces the following images
\[
\xymatrix{
{\rm id}
\ar@{|-{>}}[d]	\ar@{|-{>}}[r] & 	\bx\mapsto (0,\bp^+_1,\dots,\bp^+_n) \ar@{|-{>}}[d] 
 \\
 (\mfrak{f}^*)^i \ar@{|-{>}}[r] & \bx \mapsto (0, (\bp^+_1)^{q^{i-1}} + \pi(\bp^+_2)^{q^{i-2}}+ \cdots + 
 \pi^{i-1} \bp^+_i, \dots).
}
\]	
The conclusion follows by the explicit description of the map $\Theta^+$ as in the lines below \eqref{e.Nnp}
\end{proof}

We can now prove the main result of this section. This is a particular case of \cite[Theorem~1.3]{Sah} for which we give a shorter proof.

\begin{theorem}
	\label{t.Njet}
	Let $G$ be a smooth formal group scheme over $R$. Then for all $n$ we 
	have
	$$
	N^nG \simeq J^{n-1}(N^1G).
	$$
\end{theorem}
\begin{proof}
	Let $\AA=\Spf(R\langle\bx\rangle)$ be an \'etale coordinate system around the identity section of $G$. We have seen in \eqref{e.NXNA} and Example \ref{ex.NnG} that 
	$N^{n}G = N^{n}\AA =\Spec N_{n}A$ with $N_{n}A \simeq R\langle\bp^+_1,\dots , \bp^+_n\rangle$. Now by Lemma \ref{l.fi} the lateral Frobenius satisfies
\[(\mfrak{f}^*)^i(\bp^+_{1}) = (\bp^+_{1})^{q^{i-1}} + \pi (\bp^+_{2})^{q^{i-2}} + 
\cdots + \pi^{i-1} \bp^+_{i} 
\]
	for all $i=0,\dots , n$. Hence by Theorem \ref{t.coord} 	we have $N_{n}A \simeq J_{n-1}(N_{1}A)$ and we are done.
\end{proof}

Here we discuss examples for $G= \hG$ and $\hGm$ in the context of 
Theorem \ref{t.Njet} 

\begin{examples}\label{ex.NnG}
\begin{enumerate}
	\item \label{exNnG1}
		Assume $G=\hG=\Spf(R\langle x\rangle)$ with the comultiplication mapping $x$ to $x\otimes 1 +1\otimes x $. Then $ J^1\hG=\Spf(R\langle x,x'\rangle)$ where the group law is described by 	
		\begin{align*}
		x&\mapsto x\otimes 1 +1\otimes x\ , \\ x'&\mapsto x'\otimes 1+1\otimes x' +C_\pi(x\otimes 1,1\otimes x)
		\end{align*} with $C_\pi(X,Y) = \frac{X^q + Y^q -(X+Y)^q}{\pi}\in \Ou[X,Y]$. Hence $N^1\hG=\Spf(\langle x'\rangle)= \hG$. By Theorem \ref{t.Njet} and equation \eqref{e.jetGa} one concludes $N^n\hG=J^{n-1}(\hG)= \wh{\WW}_{n-1}$. \label{un1} 
		\item \label{un2} Assume $G=\hGm= \Spf(R\langle x,y \rangle/(xy-1 )$ with the comultiplication mapping $x$ to $x\otimes x$. Then $J^1\hGm=\Spf(R\langle x,y,x',y'\rangle/(xy-1,\delta(xy) ) = \Spf(R\langle x,x^{-1},x'\rangle)$ with the group law described by
		\begin{align*}
		x&\mapsto x\otimes x\ , \\ x'&\mapsto x'\otimes x^q+x^q\otimes x'+\pi x'\otimes x' \ .
		\end{align*} 
		Hence	$N^1\hGm=\Spf\left(R\langle x',y'\rangle / (x'+y'+\pi x'y')\right) = \Spf(R\langle x'\rangle) $
		and the group law on the latter maps $x'$ to $x'\otimes 1+1\otimes x'+\pi x'\otimes x'$, i.e., $N^1\hGm=\GG^\form_m\{1\} $ as formal group schemes. Now, $\GG^\form_m\{1\} $, as formal group law, has invariant differential $(1+\pi T)^{-1}dT$ and the corresponding logarithm is 
	\[\pi^{-1}\log(1+\pi T)=\sum_{j\geq 1} \frac{(-\pi)^{j-1}}{j}T^j=T+\sum_{j\geq 2} a_jT^j\in K[\![T]\!].
		\] 
Assume $p\geq e+1$. We prove that $\pi^{-1}\log(1+\pi T)\in \Ou[\![T]\!]$ and indeed in $\Ou\langle T\rangle$. It suffices to check that $v_\pi(a_j)\geq 0$ tends to infinity as $j$ tends to infinity.
Let $r>0$ and note that 
$v_\pi(a_{p^r})	=
p^r-1-er\geq 0$ since 
\[ p^r-1=(p-1)(p^{r-1}+\dots +1) \geq er.
\] 	
Further	$v_\pi(a_{p^r})<v_\pi(a_{p^{r+1}})$ and for $p^r\leq j<p^{r+1}$ we have 
\[v_\pi(a_{p^r}) = p^r-1-er
\leq j-1-v_\pi(j)=v_\pi(a_j) .	\]
Hence $\pi^{-1}\log(1+\pi T)\in \Ou\langle T\rangle$ and it defines a morphism 
of formal group schemes $ \GG^\form_m\{1\}\to \hG$.
It is an isomorphism under the stronger hypothesis that $p\geq e+2$. Indeed the inverse of $\pi^{-1}\log(1+\pi T)$ is \[
\pi^{-1}(\exp(\pi T)-1)=\sum_{j\geq 1} \frac{(\pi T)^{j}}{j!}\in K[\![T]\!],\]
and the $\pi$-adic valuation of the $j$-th coefficient is
\[ v_\pi(\pi^j/j!)=
j-v_\pi(j!)=j-e\cdot \frac{j-s_p(j)}{p-1}=
\frac{j(p-1-e)+es_p(j)}{p-1},
\]
where $s_p(j)$ denotes the sum of the digits in the base-$p$ expansion of $j$.
Clearly if $p\geq e+2$ this valuation tends to infinity as $j$ tends to infinity and hence $\pi^{-1}(\exp(\pi T)-1)
\in \Ou\langle T\rangle$. Then, if 
$p\geq e+2$, one concludes that $N^1 \hGm\simeq \hG$ and, with arguments as in \eqref{exNnG1}, that $N^n \hGm\simeq \wh{\WW}_{n-1}$.
\end{enumerate}	 
\end{examples}

The next result is an extension of \cite[Lemma 2.3]{bui95}.

\begin{lemma}\label{l.fn}
	Let $\mathcal F$ be a commutative formal group law over $R$ of dimension $d$. If $n(p-1)\geq e+1$ then $\mathcal F\{n\}\simeq (\hG)^d$ as formal group schemes over $R$. 
\end{lemma}
\begin{proof}
	In \cite[Lemma 2.3]{bui95} $R$ is a complete discrete valuation ring with algebraically closed residue field. The proof in our hypothesis works the same. Indeed Buium applies results in \cite{haz} that are valid for any $\ZZ_{(p)}$-algebra and the key-point is showing that the coefficients of the logarithm and exponential series of $\mathcal F\{n\}$ over $R[1/p]$ are indeed in $\pi R$. This is done by explicit estimates for the $\pi$-valuation of those coefficients.
\end{proof}
 
\begin{theorem} 
	\label{t.Nn}

Let $G$ be a smooth commutative formal group scheme of relative dimension $d$ over $\Spf(R)$. Assume $p\geq e+2$. Then there is a natural isomorphism of formal group schemes 
\[N^n G\simeq (\wh{\WW}_{n-1})^d .\]

\end{theorem}

\begin{proof}
	Let $G^\form$ be the formal completion of $G$ along the unit section $\Spf(R)\to G$. Let $\mathcal F\in R[[x_1,\dots, x_d,y_1,\dots, y_d]]$ be the formal group law on $ G^\form$, ${\mathcal F}^\phi$ the one obtained by acting the coefficients of $\mathcal F$ by $\phi$ and ${\mathcal F}^\phi\{1\}\colon =\pi^{-1}\mathcal F^\phi(\pi x_.,\pi y_.)$.
	By Lemma \ref{l.N1Gfgrouplaw} (see also \cite[Lemma 2.2]{bui95}) it is $N^1 G\simeq {\mathcal F}^\phi\{1\}$ as formal group schemes. Note that since $\phi(\pi)=\pi$ it is ${\mathcal F}^\phi\{1\}=({\mathcal F}\{1\})^\phi$.
	Now by hypothesis and Lemma \ref{l.fn} we have ${\mathcal F}\{1\}\simeq (\hG)^d$. Hence ${\mathcal F}\{1\}^\phi\simeq ((\hG)^d)^\phi=(\hG)^d$ and hence $N^1 G\simeq (\hG)^d $. By Theorem \ref{t.Njet} and definition of $J^{n-1}$	it is 
	\[N^n G\simeq J^{n-1}(\hG)^d\simeq (J^{n-1}\hG)^d \simeq (\wh{\WW}_{n-1})^d . 
	\]
\end{proof}

\begin{remark}
Assume $R=\Ou$, $p>2$ and let $G$ be as in the previous theorem. Then passing to limit on $n$ we have a short exact sequence
\[
	0\to N^\infty G(k)\to J^\infty G(k)=G(\Ou) \to G(k)\to 0
\]
where $G(\Ou) \to G(k)$ is the reduction map, and we recover the fact that the kernel of the reduction map is isomorphic to $\Ou^d=W(k)^d$ as groups.
\end{remark}

\section{$p$-power torsion}\label{sec:tors}

For any formal commutative $R$-group scheme $G$ let $G[p^\nu]$ denote the kernel of the multiplication by $p^\nu$ on $G$ and let $G[p^\infty]$ be the sheaf on $\Nilp_R^{\rm op}$ such that \begin{equation}\label{e.defGp}
G[p^\infty](C)=\varinjlim_\nu G[p^\nu](C) \end{equation} 
for any $C$ in $\Nilp_R$. Note that $G[p^\infty]$ is a sheaf since the above colimit commutes with equalizers \cite[04AX]{st}.
If $G$ is a formal torus or a formal abelian scheme $G[p^\infty]$ is $p$-divisible, but not in general. 

For each $\nu>0$ the closed immersions 
$G[p^\nu] \inj G[p^{\nu+1}]$ 
induce closed immersions of $\pi$-jets $J^n(G[p^\nu]) \inj J^n(G[p^{\nu+1}])$ \cite[Proposition~1.7]{bui00}, \cite[Lemma ~3.8, Theorem~3.9]{bps}.
We can naturally pass to limit on $\nu$.

\begin{lemma}\label{l.JnGlimit}
	Let the notation be as above. Then	\begin{equation*} J^n(G[p^\infty])= \varinjlim_\nu J^n(G[p^\nu]) = J^n(G)[p^\infty] 	\end{equation*}
	as sheaves on $\Nilp_R^{\rm op}$. 
\end{lemma}
\begin{proof} Recall \eqref{e.defJn} and \eqref{e.defGp}. For any $C$ in $\Nilp_R $ it is
	\[J^n(G[p^\infty])(C)= G[p^\infty](W_n(C))=
	\varinjlim_\nu G[p^\nu](W_n(C)) =
	\varinjlim_\nu J^n(G[p^\nu])(C).
	\]
Whence the first isomorphism. The second one follows by the fact that $J^n$ is left exact (being a right adjoint) and hence 
\[ 
\varinjlim_\nu J^n(G[p^\nu])(C)=
\varinjlim_\nu ((J^n G)[p^\nu])(C)= 
(J^n(G)[p^\infty])(C).
\]
	\end{proof}

Recall that $N^nF$ is defined as $\ker(J^nF\to F)$ for any sheaf of groups on $\Nilp_R^{\rm op}$. 
We then deduce from Lemma \ref{l.JnGlimit} the following result. 

\begin{lemma}\label{l.NnGlimit}
Let notation be as above. Then	\begin{equation*} N^n(G[p^\infty])= \varinjlim_\nu N^n(G[p^\nu]) = (N^nG)[p^\infty] 	\end{equation*}
		as sheaves on $\Nilp_R^{\rm op}$. 
	\end{lemma}
 \begin{proof} By Theorem \ref{t.adj0} the functor $N^n$ is left exact; hence $(N^nG)[p^\nu]=N^n(G[p^\nu])$ and
 \[ (N^nG)[p^\infty]:=\varinjlim_\nu (N^nG)[p^\nu]=\varinjlim_\nu N^n(G[p^\nu]).
 \]
 Further, $(N^nG)[p^\infty]=\ker((J^nG)[p^\infty]\to G[p^\infty])
= \ker(J^n(G[p^\infty])\to G[p^\infty])=N^n(G[p^\infty])$.
 \end{proof}

We say that a commutative formal $R$-group scheme is \emph{triangular} if it admits a finite filtration by formal subgroup schemes whose successive quotients are isomorphic to $\hG$. It is called \emph{triangular of level $0$} in \cite[p. 322]{bui95}. The formal $R$-group scheme $\wh{\WW}_n$ is triangular.

\begin{lemma} 
	\label{l.triang1}
	If $H$ is a triangular formal $R$-group scheme then $H[p^\infty]=H$ as sheaves on $\Nilp_R^{\rm op}$.
\end{lemma}
\begin{proof}
	We proceed by induction on the length $m$ of the filtration. 
	The result is clearly true for $H=\hG$ since $p$ is nilpotent in any object $C$ of $\Nilp_R$.
	Assume $m>1$. Then $H$ is extension of $\hG$ by $G$ where $G$ is a triangular formal group scheme with a filtration of length $m-1$. One concludes by induction hypothesis.
\end{proof}
Moreover, we have the following stronger result.

 \begin{theorem}\label{t.triang3} Assume $G$ is a smooth commutative formal $R$-group scheme and $p\geq e+2$. Then $N^n(G[p^\infty])= N^n G$ as sheaves on $\Nilp_R^{\rm op}$.
 \end{theorem}
 \begin{proof}
 	By {Theorem \ref{t.Nn}}, the formal $R$-group scheme $N^n G$ is triangular. Hence the result follows by Lemmas \ref{l.NnGlimit} and \ref{l.triang1}. 
\end{proof}

\begin{corollary}\label{c.pinf}
	Assume $p\geq e+2$. The exact sequence \eqref{eq:cwse} induces an exact sequence of formal groups
	\begin{equation} \label{e.pinf}
	0 \map N^n G [p^\infty] \map J^n G[p^\infty] \map G[p^\infty]\map 0,
	\end{equation}
	for any $n\ge 0$.
\end{corollary}
\begin{proof}
	Only the right exactness needs to be proved. Let $s$ be a section of $G$ such that $p^ms=0$. It lifts to a section $s'$ of $J^nG$. Now $p^ms'$ comes from a section of $N^nG$ and thus is $p$-power torsion by the previous theorem. Hence $s'$ is $p$-power torsion.
\end{proof}

We are now ready to proof the Main Theorem stated in the introduction.
\begin{theorem}\label{th.WJG}
 Assume $p\geq e+2$. 
	Given a smooth commutative formal group scheme $G$ of relative dimension $d$ over 
	$R$, for any positive integers $n$ the natural morphism $J^nG\to G$ gives an exact sequence
	\[
	0 \map (\wh{\WW}_{n-1})^d \map J^n G[p^\infty] \map G[p^\infty] \map 0
	\]
as sheaves on $\Nilp_R^{\rm op}$.
\end{theorem}

\begin{proof}
	This follows directly by applying Theorems \ref{t.Nn}, \ref{t.triang3}
	and Corollary \ref{c.pinf}.	
\end{proof}
 
 We can now conclude the study of Examples \ref{ex.NnG}.
 
\begin{examples}\label{ex.NGp}
	\begin{enumerate}
		\item Assume $ G=\hG$. We have seen in Example \ref{ex.NnG}\eqref{un1} that $N^1\hG=\Spf(R\langle x'\rangle)$ is isomorphic to $\hG=\Spf(R\langle x\rangle)$ as formal group scheme mapping $x$ to $x'$ on algebras. In particular, for any $\nu\geq 1$, it is 
		\[
		N^1\hG[p^\nu]\cong \hG[p^\nu]. 
		\]
 This result can be checked directly. Indeed $\hG[p^\nu]=\Spf(R\langle x\rangle/(p^\nu x))$ and hence
		\[
		J^1\hG[p^\nu]=\Spf\left( \frac{R\langle x,x'\rangle}{(p^\nu x,\delta(p^\nu x))}\right)\ ,\quad N^1\hG[p^\nu]=\Spf\left(\frac{R\langle x'\rangle}{(p^{\nu} x')}\right)\\ ,
		\]
	since $(p^\nu x)^q+\pi\delta(p^\nu x)=\phi(p^\nu x)=p^\nu\phi(x)=p^\nu(x^q+\pi x')$. 
 Passing to limit on $\nu$ we get an isomorphism 
$	 N^1\hG[p^\infty]\cong \hG[p^\infty]$.
\item Assume $ G=\hGm$. We have seen in Example \ref{ex.NnG}\eqref{un2} that if $p\geq e+2$ it is $N^1\hGm\simeq \hG$. Hence $N^1\hGm[p^\infty]\simeq \hG[p^\infty]$.
\end{enumerate} 	
\end{examples}

\appendix
	
\section{Prolongation sequences of algebras}\label{sec:app}

Let for this section $R$ be a flat $\Ou$-algebra with a fixed $\pi$-derivation $\dd$. In particular, it has a lift of Frobenius. Let $u\colon B\to C$ be a morphism of $R$-algebras. Recall that if $u\colon B\to C$ is a morphism of $\Ou$-algebras, a $\pi$-derivation relative to $u$ is a map of sets $\partial\colon B\to C$ such that $\partial(0)=0=\partial(1)$ and for any $x,y\in B$
\begin{align}
\label{a1}	
\partial (x+y)&=\partial (x)+\partial(y)+\frac{ u(x)^q+u(y)^q-u(x+y)^q }{\pi}\\
\label{a2}	
\partial (xy)&= u(x)^q\partial(y)+u(y)^q\partial(x)+\pi \partial(x)\partial(y) \ .
\end{align}
Associated to $\partial$ there is a lift of Frobenius (relative to $u$) $\Psi\colon B \map C$ given by $\Psi(x)= u(x)^q + \pi \partial (x)$ (see \cite[\S 3.1]{borsah19}, \cite[\S~1]{bps}).

Fix a positive integer $r$. Let $B_n = R[\bxx_0,\dots, \bxx_n]$, where for each $i\geq 0$, $\bxx_i$ denotes the $r$-tuple 
of variables $x_{i,1},\dots , x_{i,r}$ and denote by $u\colon B_{n}\to B_{n+1}$ the natural inclusions. Fix a prolongation sequence
\begin{equation}\label{e.Bprol}
B=B_0\xrightarrow{u,\partial}B_1 \xrightarrow{u,\partial}B_2\to ...
\end{equation}
i.e., for any $n$ we fix a $\pi$-derivation (relative to $u$) $\partial\colon B_n\to B_{n+1}$ such that
\begin{itemize}
	\item $\partial$ is compatible with the $\pi$-derivation of $R$;
	\item $\partial\circ u=u\circ \partial$ (and thus we avoided using subscripts).
\end{itemize}
 The lift of Frobenius is then the homomorphism \begin{equation*}\label{e.psi}
\Psi\colon B_n\to B_{n+1}, \qquad \bxx_i\mapsto \bxx_i^q+\pi \partial \bxx_i,
\end{equation*}
where we used a compact notation instead of writing $x_{i,j}\mapsto x_{i,j}^q+\pi \partial x_{i,j}$ for all $i\leq n$ and $1\leq j\leq r$.
Note that we can handle all $B_n$ together by introducing $B_\infty = R[\bxx_0, \bxx_1,\dots]=\bigcup_n B_n$ and again denoting by $\partial$ the induced $\pi$-derivation associated with the identity on $B_\infty$ and by $\Psi$ the associated lift of Frobenius.

On the other hand, starting with $B=R[\bxx_0]$ we have a prolongation sequence 
\begin{equation}\label{e.B0prol}
B=B_0\xrightarrow{u,\delta}J_1B_0 \xrightarrow{u,\delta}J_2B_0\to ...
\end{equation}
with $J_nB =R[\bxx_0,\bxx_0',\dots, \bxx_0^{(n)}]$ and $\bxx_0^{(i+1)}=\delta \bxx_0^{(i)}$.
Let \[\Phi\colon J_nB \to J_{n+1}B ,\qquad \bxx_0\mapsto \bxx_0^q+\pi \bxx_0'
\] be the corresponding lift of Frobenius and define $J_\infty B = \bigcup_n J_nB $. 
 
By \cite[(2.9)]{bps} the restriction on the first component $W(J_\infty B)\to J_\infty B $ admits a homomorphic section $\exp$ such that the following diagram
\begin{equation}\label{d.exp}
\xymatrix{
	W(J_\infty B) \ar[d]\ar[r]^w & \prod_{i\in \NN}B_\infty\\
	J_\infty B\ar@/^/[u]^{\exp} \ar[ur]_{ ({\rm id},\Phi,\Phi^2,\dots)}&
}
\end{equation}
commutes, with $w$ the ghost map of ramified Witt vectors. 
Let $\exp(\bxx_0)=(\bz_0,\bz_1,\bz_2,\dots)$ so that $\bz_0=\bxx_0$, $\bz_1=\ \bxx_0'$ and for $n>1$
\begin{equation*}\label{e.condition0}
\Phi^n(\bxx_0) = \Phi^n(\bz_0)= \bz_0^{q^n}+ \pi \bz_1^{q^{n-1}} + \cdots + \pi^n \bz_n.
\end{equation*}
We will show in Theorem \ref{t.coord} that if the indeterminates $\bxx_0,\bxx_1,\dots $ satisfy the analogous property for $\Psi$, i.e., 
\begin{equation}\label{e.condition}
\Psi^n(\bxx_0) = \bxx_0^{q^n}+ \pi \bxx_1^{q^{n-1}} + \cdots + 
\pi^n \bxx_n,
\end{equation}
 then there is a unique isomorphism between the prolongation sequences \eqref{e.Bprol} and \eqref{e.B0prol}.

We start with a technical result. 
For brevity, let us define for any $n>1$ the following polynomial in $2n-2$ indeterminates
\[H_n(x_0,\dots, x_{n-2};y_0,\dots,y_{n-2}):=\!\sum_{i=0}^{n-2}\left[\sum_{j=1}^{q^{n-1-i}}\!
\pi^{i+j} \C{q^{n-1-i}}{j} x_{i}^{ q(q^{n-1-i}-j)}y_{i}^j\right]
\]
 
\begin{lemma}
	\label{l.coord1}
Let $B_*$ be the prolongation sequence in \eqref{e.Bprol} and assume that it satisfies \eqref{e.condition} for any $n$. Then 
	\[
	\bxx_n= \partial \bxx_{n-1} + H_n(\bxx_0,\dots,\bxx_{n-2};\partial \bxx_0,\dots, \partial \bxx_{n-2} )
	\]
	and hence $\bxx_n-\partial \bxx_{n-1}\in \Ou[\bxx_0,\dots,\bxx_{n-2},\partial \bxx_0,\dots, \partial \bxx_{n-2}]$ with trivial constant term.
	Similarly, in $J_\infty B$ it is
\[
\bz_n= \delta \bz_{n-1} + H_n(\bz_0,\dots,\bz_{n-2};\delta \bz_0,\dots, \delta \bz_{n-2} )
\]	 
\end{lemma}
\begin{proof}
The second assertion is \cite[Proposition 2.10]{bps}.
The proof of the first one is similar and we write below the main steps:
	\begin{eqnarray*}
		\sum_{i=0}^n \pi^i \bxx_i^{q^{n-i}} &=& \Psi^n(\bxx_0)\\
		&=&\Psi (\Psi^{n-1}(\bxx_0)) \\
		&=& \Psi\left( \sum_{i=0}^{n-1}\pi^i \bxx_i^{q^{n-1-i}} \right)\\
		&=& \sum_{i=0}^{n-1} \pi^i \Psi(\bxx_i)^{q^{n_i}}\\
		&=& \sum_{i=0}^{n-1} \pi^i\left(\bxx_i^q + \pi \partial\bxx_i\right)^{q^{n_i}}\\
		&=& \sum_{i=0}^{n-1} \pi^i \left[ \bxx_i^{q^{n-i}} + \sum_{j=1}^{q^{n_i}}
		\C{q^{n_i}}{j}\bxx_i^{(q^{n_i}-j)q}\pi^j(\partial \bxx_i)^j \right] \\
		&=& \sum_{i=0}^{n-1}\pi^i \bxx_i^{q^{n-i}}+ \sum_{i=0}^{n-1} \left[
		\sum_{j=1}^{q^{n_i}}\pi^{i+j} \C{q^{n_i}}{j} \bxx_i^{(q^{n_i}-j)q}(\partial \bxx_i)^j\right]
	\end{eqnarray*}
	with $n_i=n-1-i$. 
	Hence cancelling the common terms on both sides of the above equality we get
	\[
	\pi^n\bxx_n = \pi^n\partial \bxx_{n-1} + \sum_{i=0}^{n-2}\left[\sum_{j=1}^{q^{n-1-i}}
	\pi^{i+j} \C{q^{n-1-i}}{j} \bxx_{i}^{ q(q^{n-1-i}-j)}(\partial \bxx_{i})^j\right] 
	\]
	and one can divide by $\pi^n$. 
\end{proof}

We now prove that prolongation sequences $B_*$ as above satisfying condition \eqref{e.condition} for all $n$ are unique up to unique isomorphism. 

\begin{theorem} 	\label{t.coord}
Let $B_*$ be the prolongation sequence in \eqref{e.Bprol}. Assume that the indeterminates $\bxx_i$ satisfy \eqref{e.condition} for all $n$ and let $\bz_i\in J_\infty R[\bx_0]$ be the elements defined just below \eqref{d.exp}.
 Then we have
 \begin{enumerate}[label={(\roman*)}]
 	\item \label{t.1} The inclusion $R[\bz_0,\dots, \bz_n]\to J_n R[\bx_0]$ is an isomorphism for any $n$.
 \item For any $n\geq 0$ the $R$-algebra homomorphism $h_n\colon J_n R[\bx_0]\to B_n,\ \bxx_0^{(i)}\mapsto \partial^i \bxx_0$ is an isomorphism and the following square
\[\xymatrix{
R\left[\bxx_0,\dots, \bxx_0^{(n)}\right] \ar[rr]^{(u,\delta)}\ar[d]^{h_n}&& R\left[\bxx_0,\dots, \bxx_0^{(n+1)}\right] \ar[d]^{h_{n+1}}
\\
R\left[\bxx_0,\dots, \bxx_n\right] \ar[rr]^{(u,\partial)} && R\left[\bxx_0,\dots, \bxx_{n+1}\right] 
}
\]
commutes, where $u$ denotes the inclusion map on both levels.
 	\end{enumerate}
 \end{theorem}
 
 \begin{proof}
 The first assertion was proved in \cite[Lemma 2.20]{bps} with $\bz_i=P_i(\bx)$.
Commutativity of the squares is immediate by definition of $h_n$.
We are then left to prove that $h_n$ is an isomorphism. Note hat $J_nR[\bx_0]=R[\bz_0,\dots,\bz_n]$ by point \ref{t.1} and $B_n=R[\bx_0,\dots,\bx_n]$. If we prove that $h_n(\bz_i)=\bx_i$, for all $0\leq i\leq n$, the result is clear.

We proceed by strong induction on the subset $\{(i,n), 0\leq i\leq n\}\subset \NN^2$ totally ordered as follows:
\[(i_1,n_1)<(i_2,n_2) \quad \text{ if } \quad n_1+i_1<n_2+i_2 \quad \text{ or } \quad n_1+i_1=n_2+i_2 \ \text{ and } \ i_1<i_2 .\]
The picture below illustrates the order.

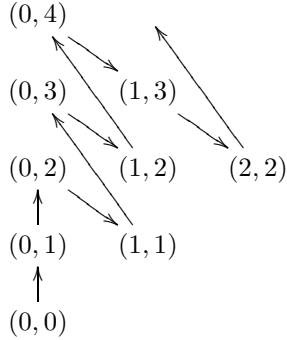
\[
\xymatrix@C=13pt@R=13pt{
(0,4)\ar[dr] & \\
(0,3)\ar[dr] & (1,3)\ar[dr] & \\
(0,2)\ar[dr] & (1,2)\ar[uul] & (2,2)\ar[uul] & \\
(0,1)\ar[u] & (1,1)\ar[uul] &  & \\
(0,0)\ar[u] & &  & }
\]

It follows immediately by definition of $h_n$ that $h_n(\bz_0)=h_n(\bx_0)=\bx_0$ and $h_n(\bz_1)=h_n(\bx_0')=\partial \bx_0=\bx_1$ for all $n\geq 0$. Hence the assertion $h_n(\bz_i)=\bx_i$ is clear for $i\leq 1$ and any $n$, in particular for the base step $(0,0)$.
Assume then $i>1$ and that $h_s(\bz_j)= \bx_j$ for any $(0,0)\leq (j,s)<(i,n)$. 
By Lemma \ref{l.coord1}, the commutativity of the above square and the induction step, we have
\begin{eqnarray*}
h_n(\bz_i)&=& h_n\left( \delta \bz_{i-1}+H_i(\bz_0,\dots,\bz_{i-2},\delta \bz_0,\dots, \delta \bz_{i-2}) \right )\\
&=& h_n(\delta \bz_{i-1})+
H_i(h_n(\bz_0),\dots,h_n(\bz_{i-2}),h_n(\delta \bz_0),\dots, h_n(\delta \bz_{i-2})) \\
&=& \partial(h_{n-1} (\bz_{i-1}))+
H_i(\bx_0,\dots,\bx_{i-2},\partial h_{n-1}( \bz_0),\dots, \partial h_{n-1}( \bz_{i-2})) \\
&=& \partial \bx_{i-1}+
H_i(\bx_0,\dots,\bx_{i-2},\partial\bx_0,\dots, \partial \bx_{i-2}) \\
&=& \bx_i .
\end{eqnarray*}
 \end{proof}

\bibliographystyle{plain}

\end{document}